\RequirePackage{fix-cm}
\documentclass[smallextended]{svjour3}       
\smartqed  
\usepackage[T1]{fontenc}
\usepackage[utf8]{inputenc}
\usepackage[russian,english]{babel}
\usepackage{graphicx}
\usepackage{amsmath,amssymb,amsfonts,stmaryrd,mathtools}
\usepackage[activate={true,nocompatibility},spacing,kerning]{microtype}
\usepackage{xcolor}
\definecolor{darkred}{rgb}{0.5,0,0}
\definecolor{gcolor}{rgb}{0.004,0.396,0.741}	
\usepackage[pdftex,
            colorlinks,
            hyperfootnotes=false,
            pdffitwindow=true,
            plainpages=false,
            pdfpagelabels=true,
            pdfpagemode=UseOutlines,
            pdfpagelayout=TwoPageRight,
            pdftitle={Closed-Form Integration of Newton Potentials between Two Cubes},
            pdfauthor={Folkmar Bornemann, TUM},
            hyperindex,
            setpagesize=false,
            filecolor=purple,
            urlcolor=darkred,
            citecolor=gcolor,
            linkcolor=gcolor]{hyperref}

\journalname{Journal}
\input{macros.sty}
\begin{document}


\title{The Challenge of Sixfold Integrals: The Closed-Form Evaluation of Newton Potentials between Two Cubes
}

\titlerunning{Closed-Form Integration of Newton Potentials between Two Cubes}        

\author{Folkmar Bornemann}


\institute{Folkmar Bornemann \at
              Department of Mathematics\\ 
              Technical University of Munich \\
              \email{bornemann@tum.de}}

\date{April 25, 2022}

\maketitle

\begin{abstract}
The challenge of explicitly  evaluating, in elementary closed form, the weakly singular sixfold integrals for potentials and forces between two cubes has been taken up at various places in the mathematics and physics literature. It created some strikingly specific results, with an aura of arbitrariness, and  a single intricate general procedure due to Hackbusch. Those scattered instances were mostly addressing the problem heads on, by successive integration while keeping track of a thicket of primitives generated at intermediate stages.

In this paper we present a substantially easier and shorter approach, based on a Laplace transform of the kernel. We clearly exhibit the structure of the results as obtained by an explicit algorithm, just  computing with rational polynomials. The method extends, up to the evaluation of single integrals, to higher dimensions. Among other examples, we easily reproduce Fornberg's startling closed form solution of Trefethen's two-cubes problem and Waldvogel's  symmetric formula for the Newton potential of a rectangular cuboid.
\end{abstract}

{\small}

\section{Introduction}\label{sec:intro}

Here are three openings for this paper, looking at some remarkable spotlights of closed form evaluations of Newton potentials and forces in reverse chronological order. Interestingly, the protagonists started their work each time from scratch, completely independent of each other.
\begin{itemize}
\item[1.] In Oct. 2006, as yet another instance of his “10-digit problems”, Nick Trefethen cooked up the following challenge for the graduate students in Oxford’s Numerical Analysis “Problem Solving Squad” \cite{Tref2}:

\begin{quote} 
{\em Two homogeneous unit cubes of unit mass attract each other gravitationally according to Newton's law with unit gravitational constant. Their centers are one unit apart, so the cubes are right up against each other, touching. 
What is the force, to ten digit accuracy?\\*[1mm]\phantom{XXX}\hfill{\small\em [slightly edited for a better fit]}}
\end{quote}

\medskip

With some background in theoretical mechanics, the challenge amounts to evaluating the weakly singular sixfold integral
\begin{equation}\label{eq:Trefethen}
F = \iiint\limits_{[1,2]\times[0,1]^2} \iiint\limits_{[0,1]^3}\frac{(x_1-y_1)\;dy_1dy_2dy_3\,\,dx_1dx_2dx_3}{\left((x_1-y_1)^2+(x_2-y_2)^2+(x_3-y_3)^2\right)^{3/2}}.
\end{equation}
As a problem in numerical analysis, using a delightful subdivision idea the challenge was nailed by graduate student Alex Prideaux, see \cite[p.~124]{Tref1}:
\[
F =0.92598\,12605\,57\cdots \,.
\]
While on sabbatical at Oxford in 2010, Bengt Fornberg learnt about the problem, thought it ought to be doable in closed form by elementary means and, after a week full of fun and intense work,  surprised everybody with\footnote{In a personal communication Fornberg kindly shared his recollections: “Playing with that was a fun about week-long episode [\,\ldots] I do not any longer have any of my paper scribbles left from when working out the successive integrals, but I recall a lot of integrations by parts (but no fancier calculus ingredients than that). I used {\em Mathematica} in `calculator mode', mostly to verify steps and to ensure typo-free bookkeeping of terms. Actually using {\em Mathematica} (at least in how I tried that) was somewhat dangerous, as multivalued functions arose, and I recall some non-relevant choices. I ended up about four days later with a pretty horrendous expression  which you, in the attached [26 terms along the lines of $-\frac13\log(4\,895\,281+1\,998\,541\sqrt6)$ and the like], can see my numerical verification of. From that point on, what remained was about another three days of `clean-up'. I was not at all sophisticated about that, mostly using the general technique of `messing around' with basic log and arctan formulas.”} 
\begin{multline}\label{eq:Fornberg}
F = \frac{1}{3} \Big(-14+2 \sqrt{2}-4 \sqrt{3}+10
   \sqrt{5}-2 \sqrt{6} + 26 \log 2 - 2\log 5 \\
   +10 \log(1+\sqrt{2})+20 \log
  (1+\sqrt{3})
   -35 \log(1+\sqrt{5})\\
   +6 \log
  (1+\sqrt{6})-2 \log
  (4+\sqrt{6}) +\frac{26 \pi }{3}-22 \arctan(2
   \sqrt{6}) \Big).
\end{multline}
Of course, to be really confident, it was checked against the numerical value.

\quad\; Asked by a reader whether {\em Mathematica} can do it, Michael Trott discussed the problem in the Oct. 2012 entry \cite{Trott} of his blog. He explored, by a Laplace transform technique,\footnote{In the physics literature on this topic, the Laplace transform technique was previously used by Orion Ciftja in 2010--11 for the evaluation, in elementary finite terms, of the electrostatic self-energy of a homogeneous unit square \cite{Ciftja2} and a homogeneous unit cube \cite{Ciftja3}.} the Newton potential between two homogeneous unit cubes with centers at a distance of $X$,  then taking the derivative at $X=1$; in this way, basically, reproducing Fornberg's solution. 

\quad\;Understanding, and simplifying, the mathematical structure of that blog entry, and of Fornberg's solution in the first place, motivated a preliminary version of this paper.

\item[2.] In 2001, with boundary element methods in mind, Wolfgang Hack\-busch~\cite{Hackbusch} devised a direct method to calculate explicit expressions for 
gravitational or electrostatic potentials, their evaluation being a task often required in weakly singular integral equations. He considered the sixfold integral
\begin{equation}\label{eq:Vvanilla}
V := \iiint\limits_Q \iiint\limits_{Q'}\frac{x_1^{n_1}x_2^{n_2}x_3^{n_3} y_1^{m_1}y_2^{m_2}y_3^{m_3}\;dy_1dy_2dy_3\,\,dx_1dx_2dx_3}{\sqrt{(x_1-y_1)^2+(x_2-y_2)^2+(x_3-y_3)^2}},
\end{equation}
where $n_j,m_j\in\N_0:=\{0,1,2,\ldots\}$ and $Q$, $Q'$ are cuboids of the form
\[
Q = [a_1,b_1]\times[a_2,b_2]\times[a_3,b_3],\qquad Q' = [a'_1,b'_1]\times[a'_2,b'_2]\times[a'_3,b'_3].
\]
The monomials in the numerator correspond, e.g., to polynomial approximations of inhomogeneous mass or charge distributions (cf. \cite{cubic}), or to the possible choice of higher order finite elements in a Galerkin method. 

\quad\; Hackbusch proceeded by successive integration, from the first integral to the sixth in the order $x_1,y_1,x_2,y_2,x_3,y_3$, while carefully controlling the structure of the primitives as linear combinations of terms of a certain form (altogether there are 18 different such forms, see \cite[Table (6.5)]{HackbuschCode}). The description of the terms, and the proofs of their recursion formulae, extend over 16 pages in his paper.
To deal with concrete cases, he assigned the actual computations to a Pascal program \cite{HackbuschCode} with about 3000 lines of code, including \TeX\ output of the results. Here is a typical result, very much resembling the look and feel of Fornberg's expression \eqref{eq:Fornberg}:
\begin{multline}\label{eq:Hackbusch}
\iiint\limits_{[0,1]^3} \iiint\limits_{[0,1]^3}  \frac{x_1x_2x_3 \;y_1^2y_2^2y_3^2\;\, dy_1dy_2dy_3\,\,dx_1dx_2dx_3}{\sqrt{(x_1-y_1)^2+(x_2-y_2)^2+(x_3-y_3)^2}} = \frac{1}{120}
-\frac{\sqrt{2}}{336}
-\frac{\sqrt{3}}{224}\\*[0mm]
+\frac{13}{560}\log( 1+\sqrt{2}) 
+\frac{1}{70}\log( 1+\sqrt{3}) -\frac{\log  \sqrt{2}}{70}
-\frac{61\pi}{13440}.
\end{multline}
As a verification, Hackbusch  checked the examples in \cite[§3.14]{Hackbusch} against numerical values; whereas we enjoy the luxury to check against his code.

\quad\;Inspired by the exploration found in Trott's blog entry~\cite{Trott}, we present a new general method
for the direct evaluation of $V$ that is much easier to derive, to describe (about 5 pages), and to code (25 lines of basic {\em Mathematica} code;\footnote{See the {\em Mathematica} notebook coming with the sources at \href{https://arxiv.org/abs/2204.02793}{arXiv:2204.02793}; it includes all the algorithmic calculations reported below and also some supplementary material.} but pen and paper would suffice for all the examples given here).

\medskip

\item[3.] In 1976,  to simplify 
MacMillan's \cite[pp. 72--80]{MacMillan} classical formula from 1930 which extends over $1\frac12$ pages, Jörg Waldvogel \cite[Eq.~(15)]{MR442257} obtained that 
\[
V(y_1,y_2,y_3):=\iiint\limits_{Q}\frac{dx_1dx_2dx_3}{\sqrt{(x_1-y_1)^2+(x_2-y_2)^2+(x_3-y_3)^2}},
\]
the potential of a homogeneous rectangular cuboid $Q$
at a point $(y_1,y_2,y_3)$, is given by a symmetric formula that has 48 terms if expanded: namely, short and crisp, writing $\rho_\delta=\sqrt{\delta_1^2+\delta_2^2+\delta_3^2}$ and $\displaystyle\sum_\text{cyc:$\delta$}$ for sums cyclic in $(\delta_1,\delta_2,\delta_2)$,\\*[-4mm]
\begin{multline}\label{eq:waldvogel}
V(y_1,y_2,y_3) \\*[1mm]=\sum_{\substack{\delta_j \in \{y_j-a_j,b_j-y_j\}\\*[1mm](j=1,2,3)}} \sum_{\text{cyc:$\delta$}}
\bigg( \delta_1\delta_2 \arctanh\frac{\delta_3}{\rho_\delta} - \frac{\delta_1^2}{2}\arctan\frac{\delta_2\delta_3}{\delta_1\rho_\delta}\bigg).
\end{multline}
By nicely exploiting symmetry and some homogeneity (via Euler's theorem), Wald\-vogel just had to evaluate a double integral.
In addition, subsequently integrating over $y_1,y_2,y_3$, he obtained  the self-energy of the unit cube,
\begin{multline}\label{eq:selfcube}
\frac{1}{2}\iiint\limits_{[0,1]^3}\iiint\limits_{[0,1]^3}\frac{dy_1dy_2dy_3\;dx_1dx_2dx_3}{\sqrt{(x_1-y_1)^2+(x_2-y_2)^2+(x_3-y_3)^2}} \\*[1mm]= \frac{1}{5}\Big(1+\sqrt{2}-2\sqrt{3}\,\Big)+ \log\Big((1+\sqrt{2})(2+\sqrt{3})\Big) -\frac{\pi}{3},
\end{multline}
a result that has repeatedly been rediscovered, without reference though; see, e.g., \cite[p.~207]{Hackbusch} and \cite[Eq.~(6)]{Ciftja3}.

\quad\; As a bonus, we show in Section~\ref{sect:waldvogel} that the method presented here allows us to straightforwardly reproduce the lovely symmetric formula \eqref{eq:waldvogel}.
\end{itemize}
The main objective of this paper is to give these three spotlights a common frame that allows an understanding of their structure as well as a comparatively simple algorithm to obtain, among other things, all previous results and formulae.

\subsection*{{\em The basic idea and structure of the results in 3D}}

 The simple idea of this paper, quite classical in multivariate integration,\footnote{As it deserves to be better known, Appendix A discusses two historic examples: one in multivariate calculus, attributed to Cauchy, and one in probability theory, due to Montroll.}  is to turn a sixfold integral such as \eqref{eq:Trefethen}  into a single integral of a product of three double integrals by means of a Laplace transform (that is, by bringing yet another integral into the fold, followed by a change of order of integration). 

Let us describe this for the evaluation of the Newton potential $V$ as given by \eqref{eq:Vvanilla}. First, inserting the following Laplace transform (where the substitution $s=\sigma^2$ simplifies things right at the beginning)
\begin{equation}\label{eq:mean}
\frac{1}{\sqrt{t}} = \frac{1}{\sqrt{\pi}}\int_0^\infty \frac{e^{-st}}{\sqrt{s}}\,ds=\frac{2}{\sqrt{\pi}} \int_0^\infty e^{-\sigma^2 t}\,d\sigma\qquad (t>0)
\end{equation}
gives, by Fubini's theorem,
\begin{multline*}
V = \frac{2}{\sqrt{\pi}} \iiint\limits_Q \iiint\limits_{Q'} x_1^{n_1}x_2^{n_2}x_3^{n_3} y_1^{m_1}y_2^{m_2}y_3^{m_3}\\*[1mm]\left(\int_{0}^\infty  e^{-\sigma^2 \left((x_1-y_1)^2+(x_2-y_2)^2+(x_3-y_3)^2\right)}\,d\sigma\right)  \,dy_1dy_2dy_3\,\,dx_1dx_2dx_3
\end{multline*}
\begin{multline}\label{eq:V}
= \frac{2}{\sqrt{\pi}}\int_{0}^\infty \prod_{j=1}^3\left( \int_{a_j}^{b_j} \int_{a_j'}^{b_j'} x_j^{n_j}y_j^{m_j} e^{-\sigma^2(x_j-y_j)^2} \,dy_j\,dx_j \right)\,d\sigma\\*[1mm]=
\frac{2}{\sqrt{\pi}}\int_{0}^\infty f_1(\sigma)f_2(\sigma)f_3(\sigma)\,d\sigma,
\end{multline}
where $f_j$ specifies the following function (with $n$, $m$ suppressed in the notation) for the parameters and variables with index $j$:
\begin{equation}\label{eq:F}
f(\sigma):= \int_{a}^{b} \int_{a'}^{b'} x^{n}y^{m} e^{-\sigma^2(x-y)^2} \,dy\,dx.
\end{equation}

Next, as explained in Section~\ref{sect:f}, by introducing the error function $\Erf(\xi)$ simply as the unique odd primitive (antiderivative) of the Gaussian $e^{-\xi^2}$, we get an algorithmic description of the factors $f(\sigma)$: they are linear combinations
of terms of the form
\[
\sigma^{-2k} e^{-\sigma^2 \delta^2} \quad\text{and}\quad \sigma^{-(2k-1)}\Erf(\sigma\delta) \qquad (k = 1,2,\ldots\,),
\]
where $\delta$ is a difference between one of the bounds of the outer integral and one of the bounds of the inner integral. Actually, this form allows right away the rather straightforward numerical evaluation of $V$ by just integrating the single integral in \eqref{eq:V}.

Finally, as explained in  Section~\ref{sect:V}, the error function serves merely as an intermediary on the way to the elementary closed form of $V$. We arrive there by a renormalization process, based on partial integration for removing cancelling singularities at $\sigma=0$ of the integrand $f_1(\sigma)f_2(\sigma)f_3(\sigma)$. In passing, this process also reduces 
the number of possible terms to eventually just three basic forms---for which, ultimately, there are entries found in a table of integrals such as the one by Prudnikov et al. \cite{MR950173} (for the sake of completeness, we give simple proofs of these entries in Appendix B).

As explained in  Section~\ref{sect:Fj}, this algorithmic construction extends also to the evaluation of force components and yields, in summary:

\begin{thm} {\em In the 3D case, the potential and also the force components are linear combinations of terms of the form\footnote{Recall the notation $\rho_\delta=\sqrt{\delta_1^2+\delta_2^2+\delta_3^2}$ that we introduced in Waldvogel's formula \eqref{eq:waldvogel}.}
\[
\frac{1}{\rho_\delta}, \quad  \arctanh\frac{\delta_3}{\rho_\delta}, \quad
\arctan\frac{\delta_2\delta_3}{\delta_1\rho_\delta}, 
\]
where $(\delta_1,\delta_2,\delta_3)$ satisfies $\delta_1\neq 0$ and is, up to a permutation, the component-wise difference between the coordinates of a vertex of $Q$ and a vertex of $Q'$; the coefficients of the terms are rational polynomials in the integral bounds.}
\end{thm}
Note that this theorem reveals a general structure that is already visible in Waldvogel's formula~\eqref{eq:waldvogel}---an example that effectively corresponds to the case that $Q'$ is shrunk to a point $(y_1,y_2,y_3)$ while the potential is properly rescaled. By rewriting the inverse hyperbolic tangent in terms of logarithms, this structure is also present in  the concrete cases \eqref{eq:Fornberg}, \eqref{eq:Hackbusch}, and \eqref{eq:selfcube}.

\subsection*{\em Extensions and limitations of the method}

When extending the class of problems to higher dimensions, or other (weakly) singular integral kernels, we hit the limits  of multiple integrals being evaluable in terms of elementary closed form expressions:

First, in dealing with force components in Section~\ref{sect:Fj}, we note the importance of factors such as $x_1-y_1$ in the numerator of \eqref{eq:Trefethen} to make the method work; by means of an example we argue that general integrals with an $r^{-3}$ kernel are highly unlikely to be evaluable in elementary finite terms at all.

Second, other dimensions than 3D are the topic of Section~\ref{sect:otherdim}. Whereas integration in elementary finite terms extends easily to the 1D and 2D cases (actually even reducing the number of possible forms in the Main Theorem from three to just two), we argue that the 4D and higher dimensional cases are highly unlikely to be evaluable in elementary finite terms at all.

 Therefore, in such cases we have to resort to the (straightforward) numerical evaluation of a single integral; with dimensions $100$ or higher posing no problem whatsoever.

\section{Evaluation of the Factor $f(\sigma)$}\label{sect:f}
Since it serves only the role of an intermediary, we normalize the error function in a slightly non-standard manner for reasons of convenience:\footnote{The standard form $\erf(\xi)$ of the error function is normalized as 
\[
\erf(\xi) := \frac{2}{\sqrt\pi}\int_0^\xi e^{-x^2}\,dx,
\]
which would yield a lot of distracting factors $\sqrt{\pi}$ in our formulae; for $\Erf(\xi)$ see \cite[§7.1]{MR2655347}.}
it is simply defined here as the unique odd primitive of the Gaussian, that is 
\begin{equation}\label{eq:Erf}
\Erf(\xi) := \int_0^\xi e^{-x^2}\,dx,
\end{equation}
from which we directly obtain the asymptotics as $\xi\to 0$, and $\xi\to\infty$, namely
\begin{equation}\label{eq:erfasymp}
\Erf(\xi) = \xi + O\big(\xi^3\big), \quad \text{and}\quad \Erf(\xi) = \frac{\sqrt{\pi}}{2} + O\big(\xi^{-1}e^{-\xi^2}\big).
\end{equation}
Writing primitives as indefinite integrals (with constants of integration determined implicitly by the given formulae), we 
get the following result:
\begin{lemma}\label{lem:AB} For $n=0,1,2,\ldots$\,, define 
\begin{align*}
A_{n}(x,y) &= u_n(x,y) e^{-(x-y)^2} + v_n(y) \Erf(x-y)\\*[1mm]
B_n(x,y) &= - \frac{u_{n+1}(x,y) }{n+1} e^{-(x-y)^2} + \frac{x^{n+1}-v_{n+1}(y)}{n+1} \Erf(x-y) ,
\end{align*}
with $u_n \in \Q[x,y]$ of degree $n-1$ and $v_n \in \Q[y]$ of degree $n$ recursively given by 
\begin{align*}
u_{n+1}(x,y) &= y\, u_n(x,y) + \frac{n}{2}u_{n-1}(x,y) -\frac{x^n}{2},\\*[1mm]
v_{n+1}(x,y) &= y\, v_n(y) + \frac{n}{2}v_{n-1}(y),
\end{align*}
and initial values taken as $u_0(x,y)=0$ and $v_0(y)=1$. Then these functions are primitives of the $n^\text{th}$-moments of $e^{-(x-y)^2}$ and $\Erf(x-y)$ wrt $x$, that is
\[
A_n(x,y) = \int x^n e^{-(x-y)^2}\,dx,\quad B_n(x,y) = \int x^n \Erf(x-y)\,dx \qquad (n\in\N_0).
\]
\end{lemma}
\begin{proof} The recursions for $u_n$ and $v_n$ combine into the single one
\begin{subequations}
\begin{align}
A_{n+1}(x,y) &= y A_n(x,y) + \frac{n}{2}A_{n-1}(x,y) - \frac{x^n}{2} e^{-(x-y)^2},\label{eq:A}\\*[1mm]
B_n(x,y) &= \frac{x^{n+1}}{n+1}\Erf(x-y) - \frac{1}{n+1} A_{n+1}(x,y),
\end{align}
\end{subequations}
with initial value $A_0(x,y) = \Erf(x-y)$. Now a straightforward direct calculation confirms inductively, for $n=0,1,2,\ldots$\,, that
\[
\partial_x A_n(x,y)=x^ne^{-(x-y)^2},\qquad \partial_x B_n(x,y)=x^n \Erf(x-y).
\]
Note that since $u_{-1}(x,y)$, $v_{-1}(y)$ and $A_{-1}(x,y)$ are multiplied by $n/2$ with $n=0$ in their recursions, any value for them will do.
\qed
\end{proof}

A concrete example of the use of this lemma, with $n=4$, can be found in Section~\ref{sect:waldvogel}, following Eq.~\eqref{eq:Fxyz}. We proceed by considering the double indefinite integrals that are the basis for calculating the factor $f(\sigma)$. 

\begin{corollary}\label{cor:f} For $n,m \in\N_0$ there holds
\[
\iint x^n y^m e^{-(x-y)^2} \,dx\,dy = p_{n,m}(x,y)e^{-(x-y)^2} + q_{n,m}(x,y) \Erf(x-y),
\]
where $p_{n,m},\, q_{n,m} \in \Q[x,y]$ are polynomials of degree $n+m$ and $n+m+1$ with
\begin{gather*}
p_{n,m}(x,y) = p_{n,m}^u(x,y) + p_{n,m}^v(x,y),\\*[1mm]
q_{n,m}(x,y) = q_{n,m}^u(x,y) + q_{n,m}^v(x,y),
\end{gather*}
which are obtained by expanding and replacing according to the following rules:
{\small\[
\begin{array}{c|c|c}
& \text{\rm expand} & \text{\rm replace $y^k$ by} \\*[1mm]\hline
\phantom{\Big|} p_{n,m}^u(x,y) \quad & \quad u_n(x,y)y^m \quad & \quad u_k(y,x)\\*[1mm]
\phantom{\Big|} q_{n,m}^u(x,y) \quad & \quad u_n(x,y)y^m \quad & \quad -v_k(x)\\*[1mm]
\phantom{\Big|} p_{n,m}^v(x,y) \quad & \quad v_n(y)y^m \quad & \quad \frac{1}{k+1}u_{k+1}(y,x)\\*[1mm]
\phantom{\Big|} q_{n,m}^v(x,y) \quad & \quad v_n(y)y^m \quad & \quad \frac{1}{k+1}(y^{k+1}-v_{k+1}(x))\\*[1mm]
\end{array}
\]}%
\end{corollary}
\begin{proof}
Using Lemma~\ref{lem:AB} we calculate a primitive, written as a  double integral,
\[
\iint x^n y^m e^{-(x-y)^2} \,dx\,dy = \int u_n(x,y)y^m e^{-(x-y)^2}\,dy  - \int v_n(y)y^m \Erf(y-x)\,dy.
\]
By expanding $u_n(x,y)y^m$ in the first integral into powers of $y$, another application of Lemma~\ref{lem:AB} 
shows that $y^k$ contributes
\[
\int y^k e^{-(x-y)^2}\,dy = A_k(y,x) = u_k(y,x) e^{-(x-y)^2} - v_k(x) \Erf(x-y)
\]
whereas expanding $v_n(y)y^m$ in the second integral shows that $y^k$ contributes
\begin{multline*}
- \int y^k \Erf(y-x)\,dy = -B_k(y,x) \\*[1mm]= \frac{1}{k+1} u_{k+1}(y,x) e^{-(x-y)^2} +\frac{1}{k+1}\left(y^{k+1}-v_{k+1}(x)\right)\Erf(x-y).
\end{multline*}
Combining the coefficients of $e^{-(x-y)^2}$ and $\Erf(x-y)$ gives the table. The claim about the polynomial degrees follows from the degrees of $u_n$ and $v_n$. \qed 
\end{proof}
Now, by substitution and rescaling we get the primitive
\begin{multline}\label{eq:fsigma}
f(\sigma;x,y):=\iint x^n y^m e^{-\sigma^2(x-y)^2} \,dx\,dy\\*[1mm] = \sigma^{-n-m-2} \left(p_{n,m}(\sigma x, \sigma y)e^{-\sigma^2(x-y)^2} + q_{n,m}(\sigma x, \sigma y) \Erf(\sigma (x-y))\right).
\end{multline}
Using this primitive we calculate the definite integral $f(\sigma)$ defined in \eqref{eq:F} as
\[
f(\sigma) = f(\sigma;b,b')-f(\sigma;a,b')-f(\sigma;b,a') + f(\sigma;a,a').
\]
Algorithmically, with the integral bounds $a,b,a',b'$ handled as variables, the construction of this expression for $f(\sigma)$ by means of Corollary~\ref{cor:f} uses simply some arithmetic of rational polynomials. Structurally, we get:
\begin{lemma}\label{lem:F}
The factor 
\[
f(\sigma)= \int_{a}^{b} \int_{a'}^{b'} x^{n}y^{m} e^{-\sigma^2(x-y)^2} \,dy\,dx
\]
is a linear combination of terms of the form\footnote{Since $\Erf(0)=0$, the second one can be dropped if $\delta=0$.}
\begin{equation}\label{eq:termF}
\sigma^{-2k} e^{-\sigma^2 \delta^2} \quad\text{and}\quad \sigma^{-(2k-1)}\Erf(\sigma\delta),
\end{equation}
where $\delta\in\Delta:=\big\{b-b',b-a',a-b',a-a'\big\}$ and the exponent $\nu$ of $\sigma^{-1}$, namely $\nu=2k$ in the first case and $\nu=2k-1$ in the second, is restricted to 
\[
1\leq \nu \leq n+m+2.
\]
The coefficient of a term belonging to $\delta = c-c'$ with $c\in\{a,b\}$ and $c'\in\{a',b'\}$ is a homogeneous polynomial of degree $n+m+2-\nu$ in $\Q[c,c']$. 
\end{lemma}
Note that Lemma~\ref{lem:F} enumerates $(n+m+2)\cdot\#\Delta$ possible terms. In concrete cases, by partial symmetries,  some of them may vanish, cancel, or combine.
\begin{proof}
By construction, $f(\sigma)$ is a linear combination of terms of the form
\[
\sigma^{-\nu} e^{-\sigma^2 \delta^2}, \qquad \sigma^{-\nu}\Erf(\sigma\delta),
\]
with $\nu=1,\ldots,n+m+2$ and $\delta$ taken from the set $\Delta$ of differences of the integral bounds; the coefficients are of the stated form. Since $f(\sigma)$ is an even function of $\sigma$, by linear independence of the even Gaussians and the odd error functions (listed by different values of the parameter $|\delta|$) over the field of rational functions, only terms that are themselves even functions of $\sigma$ can have non-zero coefficients. Hence it suffices to take only those exponents $\nu$ into account which are even in the first case and odd in the second.\qed
\end{proof}
{\small\begin{remark}
Though both forms of the terms in {\rm \eqref{eq:termF}} behave as $\propto \sigma^{-2k}$ close to $\sigma=0$, all those singularities must cancel since 
obviously, as $\sigma\to 0$,
\[
f(\sigma) = \int_{a}^{b} \int_{a'}^{b'} x^{n}y^{m} \,dy\,dx + O(\sigma^2).
\]
\end{remark}}%

\begin{example}\label{ex:1} We consider two specific cases of the integral
\[
\int_a^b\int_{a'}^{b'} x y^2  e^{-\sigma^2(x-y)^2}\,dy\,dx.
\]
Here, the exponents $\nu$ of $\sigma^{-1}$ are restricted to $1\leq \nu \leq 5$. 
The first specific case, with difference set $\Delta=\{1,2,3\}$, exhibits all of the 15 possible terms enumerated in Lemma~\ref{lem:F}:
\begin{multline}\label{eq:f15}
\int_2^3\int_0^1 x y^2  e^{-\sigma^2(x-y)^2}\,dy\,dx =
-\frac{\Erf(\sigma )}{16 \sigma
   ^5}+\frac{\Erf(2 \sigma )}{8
   \sigma ^5}-\frac{\Erf(3 \sigma
   )}{16 \sigma ^5}\\*[1mm]
   +\frac{
   \Erf(\sigma )}{\sigma ^3}-\frac{13
   \Erf(2 \sigma )}{4 \sigma
   ^3}+\frac{9  \Erf(3 \sigma )}{4
   \sigma ^3}
   +\frac{15 \Erf(\sigma
   )}{4 \sigma }-\frac{24  \Erf(2
   \sigma )}{\sigma }+\frac{81
   \Erf(3 \sigma )}{4 \sigma }\\*[1mm]
   +\frac{e^{-\sigma ^2}}{16 \sigma ^4}
   -\frac{e^{-4 \sigma ^2}}{4 \sigma ^4}
   +\frac{3 e^{-9 \sigma ^2}}{16 \sigma ^4}
   +\frac{15 e^{-\sigma ^2}}{8 \sigma ^2}
   -\frac{6 e^{-4 \sigma ^2}}{\sigma ^2}
   +\frac{27 e^{-9 \sigma ^2}}{8 \sigma ^2}.
\end{multline}
The singularities of those 12 terms that behave as $\propto\sigma^{-2}$ and $\propto\sigma^{-4}$ close to $\sigma=0$ must cancel since, as $\sigma\to 0$,
\[
 \int_2^3\int_0^1 x y^2 \, e^{-\sigma^2(x-y)^2}\,dy\,dx = \frac{5}{6} + O(\sigma^2).
 \]
In the second specific case, with  the symmetric difference set $\Delta=\{-1,0,1\}$, fewer terms than enumerated in Lemma~\ref{lem:F} have non-zero coefficients (there are no terms with $\sigma^{-4}$ and $\sigma^{-5}$):
\[
\int_0^1\int_0^1 x y^2  e^{-\sigma^2(x-y)^2}\,dy\,dx =\frac{\Erf(\sigma )}{4 \sigma
   ^3}+\frac{\Erf(\sigma )}{2 \sigma
   }-\frac{1}{2
   \sigma ^2}+\frac{e^{-\sigma ^2}}{4 \sigma ^2}.
  \]
Once again, the singularities at $\sigma=0$ cancel since
\[
\int_0^1\int_0^1 x y^2  e^{-\sigma^2(x-y)^2}\,dy\,dx = \frac{1}{6} + O(\sigma^2)\qquad (\sigma\to 0).
\]
\end{example}

%
%

\section{Evaluation of the Potential}\label{sect:V}

By \eqref{eq:V} the Newton potential \eqref{eq:Vvanilla} is given as
\[
V = \frac{2}{\sqrt\pi} \int_0^\infty g(\sigma)\,d\sigma, \qquad g(\sigma):= f_1(\sigma)f_2(\sigma)f_3(\sigma).
\]
Lemma~\ref{lem:F} shows that $g(\sigma)$ is a linear combination of terms of the form
\begin{subequations}\label{eq:origerfterms}
\begin{gather}
\sigma^{-(2\kappa+4)} e^{-\sigma^2(\delta_1^2+\delta_2^2+\delta_3^2)},\\*[1mm]
\sigma^{-(2\kappa+3)} e^{-\sigma^2(\delta_1^2+\delta_2^2)}\Erf(\sigma\delta_3),\\*[1mm]
\sigma^{-(2\kappa+2)} e^{-\sigma^2\delta_1^2}\Erf(\sigma\delta_2)\Erf(\sigma\delta_3),\\*[1mm]
\sigma^{-(2\kappa+1)} \Erf(\sigma\delta_1)\Erf(\sigma\delta_2)\Erf(\sigma\delta_3), \label{eq:nosing}
\end{gather}
\end{subequations}
where $\kappa\in\N$ and the triple $(\delta_1,\delta_2,\delta_3)$ is, up to a permutation, the component-wise difference between the coordinates of a vertex of $Q$ and a vertex of $Q'$.

All of these terms are integrable at $\sigma\to\infty$, but none of them (except the last one with $\kappa=1$) is integrable at $\sigma\to 0$. So a direct, term-wise integration of $g(\sigma)$ is bound to fail. Yet, those singularities must cancel since, as $\sigma\to0$,
\[
g(\sigma) = \iiint\limits_Q \iiint\limits_{Q'}x_1^{n_1}x_2^{n_2}x_3^{n_3} y_1^{m_1}y_2^{m_2}y_3^{m_3} \,dy_1dy_2dy_3\,\,dx_1dx_2dx_3 \;+\; O(\sigma^2).
\]
So, to remove the cancelling singularities, some kind of renormalization is called for. We write
\[
V = \lim_{\epsilon\to 0^+}\frac{2}{\sqrt\pi} \int_{\epsilon}^\infty g(\sigma)\,d\sigma
\]
and apply, following an idea explored by Trott~\cite{Trott}, repeated integration by parts as long as there are any singular terms left in the reminder integrals. 

Specifically, if there is sufficient decay at $\sigma\to\infty$, we have 
\[
\int_\epsilon^\infty \sigma^{-(\nu+1)}h(\sigma)\,d\sigma = \nu^{-1}\epsilon^{-\nu} h(\epsilon) + \nu^{-1}\int_\epsilon^\infty \sigma^{-\nu}h'(\sigma) \,d\sigma \quad (\nu\in\N, \epsilon>0).
\]
Since the functions $h(\sigma)$ will be analytic at $\sigma=0$, the sum of all the boundary terms must cancel in the end: we therefore do not calculate them in the first place. This way, by repeated application of the replacement rule
\[
\sigma^{-(\nu+1)} h(\sigma) \;\mapsto\; \nu^{-1} \sigma^{-\nu}h'(\sigma)\qquad (\nu\in\N),
\]
until no term is left to be replaced, the integrand $g(\sigma)$ is eventually transformed into a renormalized one, denoted by $\tilde g(\sigma)$, such that
\begin{equation}\label{eq:Vrenorm}
V = \frac{2}{\sqrt\pi} \int_0^\infty \tilde g(\sigma)\,d\sigma
\end{equation}
and the integral can, finally, be evaluated term-wise. 

Because the derivatives of $\exp(-\sigma^2\rho^2)$ and $\Erf(\sigma\delta)$ wrt $\sigma$ are
$-2\rho^2 \sigma e^{-\sigma^2\rho^2}$ and $\delta e^{-\sigma^2\delta^2}$, these replacement rules do not introduce any new forms of terms; for $k=1,2,\ldots$\,, there are effectively just four rules:
\begin{subequations}
\begin{multline}\label{eq:rule1}
\sigma^{-2k} e^{-\sigma^2(\delta_1^2+\delta_2^2+\delta_3^2)} \\*[1mm]\;\mapsto\; -\frac{2 (\delta_1^2+\delta_2^2+\delta_3^2)}{2 k-1}\cdot \sigma
   ^{-(2 k-2)}e^{-\sigma ^2(\delta_1^2+\delta_2^2+\delta_3^2) },
\end{multline}
\begin{multline}\label{eq:rule2}
\sigma^{-(2k+1)} e^{-\sigma^2(\delta_1^2+\delta_2^2)}\Erf(\sigma\delta_3)\\*[1mm]
\qquad\qquad \quad\; \;\mapsto\;
   -\frac{ 2(\delta_1^2+\delta_2^2)}{2k}\cdot \sigma ^{-(2 k-1)} e^{-\sigma^2(\delta_1^2+\delta_2^2)} \Erf(\sigma \delta _3)\\*[1mm]
+\frac{\delta _3 }{2k}\cdot \sigma ^{-2
   k} e^{-\sigma^2(\delta_1^2+\delta_2^2+\delta_3^2)},
\end{multline}
\begin{multline}\label{eq:rule3}
\sigma^{-2k} e^{-\sigma^2\delta_1^2}\Erf(\sigma\delta_2)\Erf(\sigma\delta_3)\\*[1mm]
 \;\mapsto\;
 -\frac{2   \delta_1^2
    }{2 k-1}\cdot\sigma ^{-(2
   k-2)}e^{-\sigma^2\delta_1^2} \Erf(\sigma\delta _2)\Erf(\sigma\delta _3)\\*[1mm]
\qquad\qquad \;\; +\frac{\delta _2}{2 k-1}\cdot\sigma ^{-(2 k-1)} e^{-\sigma^2(\delta_1^2+\delta_2^2)} \Erf(\sigma\delta _3)\\*[1mm]
 +\frac{\delta _3}{2 k-1}\cdot\sigma ^{-(2 k-1)} e^{-\sigma^2(\delta_1^2+\delta _3^2)}\Erf(\sigma\delta _2),
\end{multline}
\begin{multline}\label{eq:rule4}
\sigma^{-(2k+1)} \Erf(\sigma\delta_1)\Erf(\sigma\delta_2)\Erf(\sigma\delta_3)\\*[1mm]
  \;\mapsto\; \sum_\text{cyc:$\delta$} \frac{ \delta _1 }{2k}\cdot  \sigma ^{-2 k} e^{-\sigma^2\delta _1^2}\Erf(\sigma\delta _2) \Erf(\sigma\delta _3).
\end{multline}
\end{subequations}
We note that these rules generate terms that either vanish or feature a Gaussian factor of the form $e^{-\sigma^2\rho^2}$ with $\rho\neq 0$. Therefore, dropping the vanishing terms, we can arrange for another permutation of $(\delta_1,\delta_2,\delta_3)$ such that $\delta_1\neq0$ and the terms take one of the standard forms shown in Lemma~\ref{lem:G} below.

This way, using \eqref{eq:erfasymp}, we readily check that the resulting terms keep being integrable at $\sigma\to\infty$ and the corresponding boundary terms vanish there. 
Once again, all the replacement steps are performed just using the arithmetic of rational polynomials. To summarize, we have thus proven constructively:
\begin{lemma}\label{lem:G} The recursive application of the replacement rules \eqref{eq:rule1}--\,\eqref{eq:rule4}, until there is no term left for input, yields a renormalized integrand $\tilde g(\sigma)$ that satisfies \eqref{eq:Vrenorm}. It  is a linear combination of terms of the form
\[
e^{-\sigma^2(\delta_1^2+\delta_2^2+\delta_3^2)},\quad \sigma^{-1}e^{-\sigma^2(\delta_1^2+\delta_2^2)} \Erf(\sigma\delta_3), \quad \delta_1 e^{-\sigma^2\delta_1^2}\Erf(\sigma\delta_2)\Erf(\sigma\delta_3),
\]
where $(\delta_1,\delta_2,\delta_3)$ satisfies $\delta_1\neq 0$ and is, up to a permutation, the component-wise difference between the coordinates of a vertex of $Q$ and a vertex of $Q'$; the coefficients of the terms are rational polynomials in the integral bounds.
\end{lemma}
Now, term-wise integration is actually simple and results in elementary closed expressions. Take $\rho >0$ and $\delta_1\neq 0$. First, there is \eqref{eq:mean}, written in the form
\begin{subequations}
\begin{equation}\label{eq:I1}
\frac{2}{\sqrt{\pi}}\int_0^\infty e^{-\sigma^2\rho^2}\,d\sigma = \frac{1}{\rho},
\end{equation}
whereas, next, a table-look up gives \cite[Eq.~(2.8.5.8)]{MR950173}\footnote{In most of the examples below, the second $\log$ form is the preferred variant.\label{ft:I2}}%
\begin{multline}\label{eq:I2}
\frac{2}{\sqrt{\pi}}\int_0^\infty \sigma^{-1} e^{-\sigma^2\rho^2}\Erf(\sigma\delta)\,d\sigma = \arcsinh\frac{\delta}{\rho}  = \arctanh\frac{\delta}{\sqrt{\rho^2+\delta^2}} \\*[1mm]
= \log\rho - \log\left(- \delta + \sqrt{\rho^2+\delta^2}\,\right) =\log\left(\delta + \sqrt{\rho^2+\delta^2}\,\right) -\log\rho ,
\end{multline}
and \cite[Eq.~(2.8.19.8)]{MR950173}
\begin{equation}\label{eq:I3}
\frac{2}{\sqrt{\pi}}\int_0^\infty e^{-\sigma^2\delta_1^2}\Erf(\sigma\delta_2)\Erf(\sigma\delta_3)\,d\sigma = \frac{1}{2\delta_1}\arctan\frac{\delta_2\delta_3}{\delta_1\sqrt{\delta_1^2+\delta_2^2+\delta_3^2}}.
\end{equation}
\end{subequations}
For a simple proof of \eqref{eq:I2} and \eqref{eq:I3}, see Appendix B. 

To summarize, by Lemma~\ref{lem:G} we get, based on an algorithmic construction doing calculations in the arithmetic of rational polynomials only, the part of the Main Theorem in the Introduction which addresses the potential $V$:
\begin{corollary}\label{cor:V} Writing $\rho_\delta=\sqrt{\delta_1^2+\delta_2^2+\delta_3^2}$, the potential $V$ as defined in \eqref{eq:V} is a linear combination of terms of the form
\begin{equation}\label{eq:Vform}
\frac{1}{\rho_\delta}, \quad  \arctanh\frac{\delta_3}{\rho_\delta}, \quad
\arctan\frac{\delta_2\delta_3}{\delta_1\rho_\delta}, 
\end{equation}
where $(\delta_1,\delta_2,\delta_3)$ satisfies $\delta_1\neq 0$ and is, up to a permutation, the component-wise difference between the coordinates of a vertex of $Q$ and a vertex of $Q'$; the coefficients of the terms are rational polynomials in the integral bounds.
\end{corollary}
{\small\begin{remark} 
The remarks of Hackbusch \cite[§3.15]{Hackbusch} on numerical stabilization apply here, too.
\end{remark}}%

\begin{example}\label{ex:2}  By the second concrete case in Example~\ref{ex:1} we have
\begin{multline*}
V=\iiint\limits_{[0,1]^3} \iiint\limits_{[0,1]^3}  \frac{x_1x_2x_3 \;y_1^2y_2^2y_3^2\;\, dy_1dy_2dy_3\,\,dx_1dx_2dx_3}{\sqrt{(x_1-y_1)^2+(x_2-y_2)^2+(x_3-y_3)^2}}  \\*[1mm]
= \frac{2}{\sqrt{\pi}} \int_0^\infty \left(
\frac{\Erf(\sigma )}{4 \sigma
   ^3}+\frac{\Erf(\sigma )}{2 \sigma
   }-\frac{1}{2
   \sigma ^2}+\frac{e^{-\sigma ^2}}{4 \sigma ^2}\right)^3\,d\sigma.
\end{multline*}
The renormalized integrand $\tilde g(\sigma)$ reveals itself as
\begin{multline*}
\tilde g(\sigma) = \frac{1}{120}e^{-\sigma ^2} -\frac{1}{168}e^{-2
   \sigma ^2}-\frac{3}{224}e^{-3 \sigma ^2} \\*[1mm]
   +\frac{13}{560} \sigma^{-1} e^{-\sigma ^2} \Erf(\sigma )+\frac{1}{70 }\sigma^{-1}e^{-2 \sigma ^2}\Erf(\sigma )
   -\frac{61  }{1120}e^{-\sigma ^2}  \Erf(\sigma)^2,
\end{multline*}
which is translated by table \eqref{eq:I1}--\eqref{eq:I3}, without further ado, to the value\footnote{This result agrees exactly, term by term, with expression \eqref{eq:Hackbusch}, the output of Hackbusch's code \cite{HackbuschCode} when run on this example. However, since his code works internally with floating point numbers, just recasting them at final output by rational best approximations if {\tt TolRational} is met and the denominator is bounded by {\tt MaxDenominator}, one has to adjust the default choices of those parameters to accommodate the coefficient $61/13440$ here.} 
\[
V = \frac{1}{120}-\frac{\sqrt{2}}{336
   }-\frac{\sqrt{3}}{224}+\frac{13}{560} \log
  (1+\sqrt{2})+\frac{1}{70} \log
  (1+\sqrt{3})-\frac{\log \sqrt2}{70} -\frac{61 \pi
   }{13440},
\]
where the only simplification being made was using $\arctan(1/\sqrt{3})=\pi/6$.   
\end{example}


\section{Evaluation of Force Components}\label{sect:Fj}

The method of this paper extends to the evaluation of the force exerted by the cuboid $Q$ on $Q'$. According to basic theoretical mechanics its component in the $j^\text{th}$ coordinate direction is, by taking the derivative of the point potential under the integral sign,
\begin{equation}\label{eq:FjVanilla}
F_j := \iiint\limits_Q\iiint\limits_{Q'}\frac{(x_j-y_j)\,x_1^{n_1}x_2^{n_2}x_3^{n_3} y_1^{m_1}y_2^{m_2}y_3^{m_3}\;dy_1dy_2dy_3\,\,dx_1dx_2dx_3}{\left((x_1-y_1)^2+(x_2-y_2)^2+(x_3-y_3)^2\right)^{3/2}}.
\end{equation}
As in the case of the potential, the monomials in the numerator accommodate, e.g., inhomogeneous materials (cf. \cite{cubic} and Section~\ref{sect:waldvogel}) or higher order finite elements. Inserting the derivative of the Laplace transform \eqref{eq:mean}, that is
\begin{equation}\label{eq:var}
\frac{1}{t^{3/2}} = \frac{2}{\sqrt{\pi}} \int_0^\infty 2\sigma^2 e^{-\sigma^2 t}\,d\sigma\qquad (t>0),
\end{equation}
gives, by Fubini's theorem and in complete analogy to \eqref{eq:V}, a representation by a single integral, namely
\begin{equation}\label{eq:Fj}
F_j = \frac{2}{\sqrt{\pi}} \int_0^\infty f_j^*(\sigma) \prod_{\substack{k=1\\*[0.5mm]k\neq j}}^3f_k(\sigma)\,d\sigma \qquad (j=1,2,3).
\end{equation}
Here, $f_k$ is defined as above and $f_j^*$ specifies the following function (with $n$, $m$ suppressed in the notation) for the parameters and variables with index $j$:
\begin{equation}\label{eq:fstar}
f^*(\sigma) := 2\sigma^2 \int_{a}^{b} \int_{a'}^{b'} x^{n}y^{m} (x-y)e^{-\sigma^2(x-y)^2} \,dy\,dx.
\end{equation}
 Expressing $f^*(\sigma)$ by a primitive, a direct application of Corollary~\ref{cor:f} would give
\[
\iint x^n y^m (x-y) e^{-(x-y)^2} \,dx\,dy = p^*_{n,m}(x,y) e^{-(x-y)^2} + q^*_{n,m}(x,y) \Erf(x-y)
\]
where $p^*_{n,m},\, q^*_{n,m} \in \Q[x,y]$ are polynomials of degree $n+m+1$ and $n+m+2$. 

Actually, and most important for the method to succeed after all, the leading terms of the polynomials constructed in this way cancel and their degrees are effectively reduced by two: 

\begin{lemma}\label{lem:reduction} The polynomial $p^*_{n,m}$ has degree\footnote{A negative degree indicates that a polynomial is zero.} $n+m-1$,  $q^*_{n,m}$ has degree $n+m$.
\end{lemma}
\begin{proof} By Lemma~\ref{lem:AB}, and the recursion \eqref{eq:A}, we have
\begin{multline*}
\int x^n (x-y) e^{-(x-y)^2} \,dx = A_{n+1}(x,y)-y A_n(x,y)\\*[1mm]
 = \frac{n}{2}A_{n-1}(x,y) -\frac{x^n}{2} e^{-(x-y)^2}\qquad\\*[1mm]
 = \left(\frac{n}{2} u_{n-1} (x,y) - \frac{x^{n}}2\right) e^{-(x-y)^2} - \frac{n}{2}v_{n-1}(y) \Erf(y-x),
\end{multline*}
where $u_{n-1}(x,y)\in\Q[x,y]$ has degree $n-2$ and $v_{n-1}\in\Q[y]$ degree $n-1$. Now, as in the proof of
Corollary~\ref{cor:f}, yet another application of Lemma~\ref{lem:AB} gives
\begin{multline*}
\iint x^n y^m (x-y) e^{-(x-y)^2} \,dx\,dy \\*[1mm]
= \frac{n}{2}\int y^{m} u_{n-1} (x,y)  e^{-(x-y)^2} \,dy -\frac{x^n}{2}\int y^m e^{-(x-y)^2} \,dy\\*[1mm]
- \frac{n}{2}\int y^m v_{n-1}(y) \Erf(y-x)\,dy\\*[1mm]
= p^*_{n,m}(x,y) e^{-(x-y)^2} + q^*_{n,m}(x,y) \Erf(x-y).
\end{multline*}
The three integrals wrt $y$ contribute to the polynomials $p^*_{n,m}$ and $q^*_{n,m}$ with polynomials of degrees according to the table:
{\small\[
\begin{array}{c|c|c}
\text{integral\;}& \text{\rm \;degree of contribution to $p^*_{n,m}$\;} & \text{\rm \;degree of contribution to $q^*_{n,m}$\;} \\*[1mm]\hline
\,\text{1st}\phantom{\Big|} & \quad n+m-3 & n+m-2\\*[1mm]
\text{2nd} & \quad n+m-1 & n+m\\*[1mm]
\text{3rd} & \quad n+m-1 & n+m\\*[1mm]
\end{array}
\]}%
Hence $p^*_{n,m}$ has degree $n+m-1$ and  $q^*_{n,m}$ degree $n+m$.\qed
\end{proof}
Now, by substitution and rescaling we get the primitive
\begin{multline*}
f^*(\sigma;x,y):=2\sigma^2\iint x^n y^m (x-y)e^{-\sigma^2(x-y)^2} \,dx\,dy\\*[1mm] = \sigma^{-n-m-2} \left(2\sigma \,p^*_{n,m}(\sigma x, \sigma y)e^{-\sigma^2(x-y)^2} + 2\sigma\, q^*_{n,m}(\sigma x, \sigma y) \Erf(\sigma (x-y))\right).
\end{multline*}
This is exactly the same form as \eqref{eq:fsigma} for $f(\sigma;x,y)$, with just the replacements 
\begin{gather*}
p_{n,m}(\sigma x, \sigma y) \;\mapsto\; 2\sigma \,p^*_{n,m}(\sigma x, \sigma y),\\*[1mm]
q_{n,m}(\sigma x, \sigma y) \;\mapsto\; 2\sigma \,q^*_{n,m}(\sigma x, \sigma y).
\end{gather*}
Since the polynomial degrees wrt $\sigma$ are kept invariant this way (the decrease in the degree by $1$ in $p^*_{n,m}$ and $q^*_{n,m}$, as compared to $p_{n,m}$ and $q_{n,m}$, is compensated by the factor $2\sigma$), both of the factors $f(\sigma;x,y)$ and $f^*(\sigma;x,y)$ enjoy exactly the same structure.

Since we have used only that particular structure of $f(\sigma;x,y)$, Lemmas~\ref{lem:F} and \ref{lem:G} as well as Corollary \ref{cor:V} extend to the factor $f^*(\sigma)$ and to the evaluation of the force compontents $F_j$---with {\em literally} the same statements and proofs. This finishes, in particular, the proof of the Main Theorem stated in the Introduction.

\begin{example}\label{ex:TwoCubes} The evaluation of Trefethen's two-cubes problem \eqref{eq:Trefethen} starts with
\[
 F = \!\!\!\! \iiint\limits_{[1,2]\times[0,1]^2} \iiint\limits_{[0,1]^3}\frac{(x_1-y_1)\;dy_1dy_2dy_3\,\,dx_1dx_2dx_3}{\left((x_1-y_1)^2+(x_2-y_2)^2+(x_3-y_3)^2\right)^{3/2}}
  = \frac{2}{\sqrt{\pi}}\int_0^\infty \!\!g(\sigma)\,d\sigma,
 \]
where $g(\sigma) = f^*(\sigma) f(\sigma)^2$ factorizes into
\[
f^*(\sigma) = 2\sigma^2 \int_1^2\int_0^1 (x-y) e^{-\sigma^2(x-y)^2}\,dy\,dx = \frac{2\Erf(\sigma)}{\sigma}-\frac{\Erf(2\sigma)}{\sigma}
\]
and
\begin{equation}\label{eq:funit}
f(\sigma) = \iint\limits_{[0,1]^2} e^{-\sigma^2(x-y)^2} \,dy\,dx = \frac{ 2\Erf(\sigma )}{\sigma
   }-\frac{1}{\sigma
   ^2}+\frac{e^{-\sigma ^2}}{\sigma ^2}.
\end{equation}
The renormalized integrand $\tilde g(\sigma)$ reveals itself as an expression with 15 terms,
\begin{multline*}
\tilde g(\sigma) = \frac{2}{3} e^{-\sigma ^2} +\frac{4}{3} e^{-2
   \sigma ^2} -4 e^{-3 \sigma ^2} -\frac{32}{3} e^{-4
   \sigma ^2} +\frac{50}{3} e^{-5 \sigma ^2} -4 e^{-6 \sigma
   ^2}\\*[1mm]
   +\frac{10}{3}\sigma^{-1}e^{-\sigma ^2} \Erf(\sigma )   
   +\frac{20}{3}
   \sigma^{-1}e^{-2 \sigma ^2} \Erf(\sigma )
   -\frac{32}{3} \sigma^{-1}e^{-4
   \sigma ^2} \Erf(\sigma )\\*[1mm]
+\frac{8}{3} \sigma^{-1}e^{-5 \sigma ^2}
   \Erf(\sigma )
   -\frac{1}{3}\sigma^{-1}e^{-\sigma ^2}
   \Erf(2 \sigma )
   -\frac{2}{3}\sigma^{-1}e^{-2 \sigma ^2} \Erf(2
   \sigma )\\*[1mm]
   -\frac{40}{3}
   e^{-\sigma ^2} \Erf(\sigma )^2 + 32 e^{-4 \sigma ^2} \Erf(\sigma )^2+\frac{8}{3}
   e^{-\sigma ^2} \Erf(\sigma )\Erf(2 \sigma ),
   \end{multline*}
\end{example}
which is directly translated by table \eqref{eq:I1}--\eqref{eq:I3}, just combining two rationals and three rational multiples of $\log 2$ as well as using $\arctan(1/\sqrt{3})=\pi/6$, to 
\begin{multline}\label{eq:Bornemann}
F = -\frac{14}{3} +\frac{2\sqrt{2}}{3} -\frac{4\sqrt{3}}{3}  + \frac{10\sqrt{5}}{3} -\frac{2\sqrt{6}}{3} + \frac{23\log 2}{3}- \frac{4\log5}{3} \\*[1mm]
   +\frac{10}{3}\log(1+\sqrt{2})   
   +\frac{20}{3}\log(1+\sqrt{3}) 
   -\frac{32}{3} \log(1+\sqrt{5})\\*[1mm]
+\frac{8}{3} \log(1+\sqrt{6})     -\frac{1}{3}\log(2+\sqrt{5})   -\frac{2}{3}\log(2+\sqrt{6}) \\*[1mm]
   -\frac{10\pi}{9} + 8 \arctan\frac{1}{2\sqrt6} + \frac{4}{3}\arctan\frac{2}{\sqrt{6}}.
\end{multline}
{\small\begin{remark} 
By noting
\[
\arctan\frac{1}{2\sqrt6} =\frac{\pi}{2} - \arctan(2\sqrt{6}),\qquad 
\arctan\frac{2}{\sqrt{6}} =\frac{1}{2}\arctan(2\sqrt{6}),
\]
and
\[
(2+\sqrt{5})(2+\sqrt6)^2 = 2^{-3}5^{-2}(1+\sqrt5)^3(1+\sqrt6)^2(4+\sqrt6)^2,
\]
expression \eqref{eq:Bornemann} can be brought directly into Fornberg's original 14 term variant \eqref{eq:Fornberg}.
\end{remark}}%
Without the factor $x_j-y_j$ in the numerator of the integral \eqref{eq:FjVanilla}, and the thus induced reduction of the polynomial degrees by two as stated in Lemma~\ref{lem:reduction}, the method generally fails to deliver closed form solutions of a sixfold integral with an $r^{-3}$ kernel. E.g., we have
\begin{multline*}
H := \iiint\limits_{[1,2]^3}\iiint\limits_{[0,1]^3} \frac{dy_1dy_2dy_3\,\,dx_1dx_2dx_3}{\left((x_1-y_1)^2+(x_2-y_2)^2+(x_3-y_3)^2\right)^{3/2}}\\*[1mm]
= \frac{2}{\sqrt\pi}\int_0^\infty 2\sigma^2 \left(-\frac{2 \Erf(\sigma )}{\sigma }+\frac{2
   \Erf(2 \sigma )}{\sigma }+\frac{1}{2 \sigma ^2} -\frac{e^{-\sigma ^2}}{\sigma
   ^2}+\frac{e^{-4 \sigma
   ^2}}{2 \sigma ^2}\right)^3\,d\sigma.
\end{multline*}
Expansion of the latter integrand generates terms, among others, of the form
\[
\sigma^{-1} \Erf(\sigma\delta_1)\Erf(\sigma\delta_2)\Erf(\sigma\delta_3) \qquad (\delta_1,\delta_2,\delta_3\neq 0)
\]
which are not amenable to the renormalization algorithm of Section~\ref{sect:V}. Note that those terms are not even integrable at $\sigma\to\infty$, though their singularities cancel since the integral for~$H$ converges itself. If we restore integrability by one further step of partial integration, we create terms of the form
\[
 \log(\sigma)e^{-\sigma^2\delta_1^2}\Erf(\delta_2\sigma)\Erf(\delta_3\sigma),
\]
whose integrals are, cf. the argument given at the end of Appendix B, highly unlikely to be evaluable in elementary finite terms.
In this case we have to resort to numerical methods, which for the single integral wrt $\sigma$ give, quite straightforwardly,\footnote{Here, and in the numerical examples below, we apply {\em Mathematica}'s {\tt NIntegrate} command directly to the single integral at hand.} the value
\[
H = 0.24660\,45031\,79184\,67694\,\cdots.
\]

\section{Other Dimensions}\label{sect:otherdim}

\subsection*{\em The one-dimensional case}

In 1D we consider, with an additional parameter $\rho\geq 0$ taken as $\rho>0$ to ensure convergence in the case of overlapping intervals $I$ and $I'$, 
\[
V = \int_I\int_{I'} \frac{x^ny^m\,dy\,dx}{\sqrt{(x-y)^2+ \rho ^2}} = \frac{2}{\sqrt\pi} \int_0^\infty f(\sigma)e^{-\sigma^2\rho^2}\,d\sigma.
\]
With $\rho>0$, the results of Sects.~\ref{sect:f} and \ref{sect:V} apply and we get that $V$ is a linear combination of terms
of the form
\[
\frac{1}{\sqrt{\rho ^2+\delta^2}}, \quad  \log\left(\delta + \sqrt{\rho ^2+\delta^2}\right)-\log \rho,
\]
where $0\neq\delta \in\Delta$ and the coefficients are rational polynomials in $\rho$ and the integral bounds.
If the intervals don't overlap, we can take the limit $\rho\to0^+$ since the logarithmic singularities  at $\rho =0$ must cancel, i.e., the coefficients of $\log \rho $ must ultimately sum to zero in that limit. We note that, if $\rho=0$ and $\Delta\subset\Z$, the value of $V$ is a rational linear combination of $1,\log 2, \big\{\!\log|\delta|\big\}_{0\neq \delta\in\Delta}$.

\begin{example} By Example~\ref{ex:1}, with the concrete $f(\sigma)$ given in \eqref{eq:f15}, we have 
\[
V = \int_2^3\int_0^1 \frac{xy^2}{|x-y|}\,dy\,dx = \frac{2}{\sqrt\pi}\int_0^\infty f(\sigma)\,d\sigma.
\]
With $\Delta=\{1,2,3\}$ and $\rho =0$ we know in advance that the result is a rational linear combination of $1$, $\log 2$, and $\log 3$.
The renormalized integrand $\tilde g(\sigma)$ is
\begin{multline*}
\tilde g(\sigma) = -\frac{75}{16} e^{-\sigma ^2} +70 e^{-4 \sigma ^2} - \frac{1701}{16}
   e^{-9 \sigma ^2} \\*[1mm]
   + \frac{15}{4} \sigma^{-1}\Erf(\sigma ) - 24\sigma^{-1}\Erf(2 \sigma ) +\frac{81}{4}\sigma^{-1}\Erf(3 \sigma ).
\end{multline*}
Because of $\frac{15}{4} -24 + \frac{81}{4}= 0$, we can apply the log variants of entry \eqref{eq:I2} of the integration table even in the case $\rho=0$, by just formally ignoring the singular term $\log 0$, and get
\[
V = -\frac{41}{8}-24 \log (2)+\frac{81 \log (3)}{4}.
\]
\end{example}

\subsection*{\em The two-dimensional case}
In 2D we consider, with an additional parameter $\rho\geq 0$, on rectangles $R$, $R'$,
\[
V = \iint\limits_R\iint\limits_{R'} \frac{x_1^{n_1}x_2^{n_2}y_1^{m_1}y_2^{m_2}\,dy_1dy_2\,dx_1dx_2}{\sqrt{(x_1-y_1)^2+(x_2-y_2)^2+ \rho ^2}} = \frac{2}{\sqrt\pi} \int_0^\infty f_1(\sigma)f_2(\sigma)e^{-\sigma^2 \rho ^2}\,d\sigma.
\]
The results of Sects.~\ref{sect:f} and \ref{sect:V} apply and we get that $V$ is a linear combination of terms
of the form \eqref{eq:Vform}, where the $\arctan$ form does only appear if $\rho>0$.

\begin{example}\label{ex:5} As in Example~\ref{ex:2} we have
\begin{multline*}
V=\iint\limits_{[0,1]^2} \iint\limits_{[0,1]^2}  \frac{x_1x_2\; y_1^2y_2^2\;\, dy_1dy_2\,\,dx_1dx_2}{\sqrt{(x_1-y_1)^2+(x_2-y_2)^2}}  \\*[1mm]
= \frac{2}{\sqrt{\pi}} \int_0^\infty \left(
\frac{\Erf(\sigma )}{4 \sigma
   ^3}+\frac{\Erf(\sigma )}{2 \sigma
   }-\frac{1}{2
   \sigma ^2}+\frac{e^{-\sigma ^2}}{4 \sigma ^2}\right)^2\,d\sigma.
\end{multline*}
The renormalized integrand $\tilde g(\sigma)$ reveals itself as
\[
\tilde g(\sigma) = \frac{1}{12}e^{-\sigma ^2}-\frac{3}{20} e^{-2 \sigma
   ^2} + \frac{19}{120}\sigma^{-1} e^{-\sigma ^2} \Erf(\sigma ),
\]
which is translated by table \eqref{eq:I1}--\eqref{eq:I3}, without further ado, to 
\[
V = \frac{1}{12} - \frac{3\sqrt{2}}{40} + \frac{19}{120}\log(1+\sqrt2).
\]
\end{example}

\subsection*{\em The general $d$-dimensional case}

In the $d$-dimensional case we have, using vector notation,\footnote{As is common in vector calculus, we write $x^n := x_1^{n_1}\cdots \;x_d^{n_d}$, $y^m := y_1^{m_1}\cdots \;y_d^{m_d}$.}%
\begin{equation}\label{eq:Vdgeneral}
V = \int_Q\int_{Q'} \frac{x^n y^m}{\|x-y\|_2} \,dy\,dx = \frac{2}{\sqrt{\pi}}\int_0^\infty f_1(\sigma)\cdots f_d(\sigma)\,d\sigma,
\end{equation}
where $n,m\in\N_0^d$ are multi-indices and $Q$, $Q'$ are axis-parallel rectangular hypercuboids.
The renormalization algorithm, suitably supplemented to accommodate more factors, results in  a linear combination of terms of the form
\[
e^{-\sigma^2\delta_1} \Erf(\sigma\delta_2)\cdots \Erf(\sigma\delta_{2k-1}),\qquad \sigma^{-1} e^{-\sigma^2\delta_1} \Erf(\sigma\delta_2)\cdots \Erf(\sigma\delta_{2k})
\]
where $\nu=2k-1$ and $\nu=2k$ are restricted to $1\leq \nu \leq d$. 

However, as argued in the last paragraph of Appendix B, it is highly unlikely that there is an elementary closed form expression for integrals of such terms with three or more  factors of $\Erf$.
Thus, for $d\geq 4$ we do not expect the Newton potential $V$ to be evaluable in elementary finite terms. Instead, we have to resort to numerical methods for the single integral in \eqref{eq:Vdgeneral}.

\begin{example}\label{ex:6} By \eqref{eq:funit} we get\footnote{Equation \eqref{eq:Vd}, which expresses the self-energy $E_d=V_d/2$ of the unit $d$-cube by this particular single integral, was first obtained by Batle et al. \cite{Batle} in 2017 and used there to tabulate $E_d$ for $d=2,\ldots,28$ to five digit accuracy (with some errors in the last digit). These authors give closed forms for $E_2$ \cite[Eq.~(12)]{Batle} (see also \cite[Eq.~(18)]{Ciftja2}) and $E_3$ \cite[Eq.~(15)]{Batle} (see also \cite[Eq.~(6)]{Ciftja3}), and then speculate about the general case \cite[p.~55]{Batle}: “However, any analytic calculation is expected
to be very lengthy and challenging even for the simplest case of $d=4$.” }
\begin{equation}\label{eq:Vd}
V_d = \int_{[0,1]^d}\int_{[0,1]^d}\frac{dy\,dx}{\|x-y\|_2} = \frac{2}{\sqrt\pi}\int_0^\infty\left(\frac{ 2\Erf(\sigma )}{\sigma
   }+\frac{e^{-\sigma ^2}-1}{\sigma ^2}\right)^d\,d\sigma.
\end{equation}
For $d=4$ the renormalized integrand $\tilde g(\sigma)$ reveals itself as
\begin{multline*}
\tilde g(\sigma) = \frac{32}{315} e^{-\sigma ^2} +\frac{136}{105} e^{-2 \sigma
   ^2} -\frac{48}{35} e^{-3 \sigma ^2}-\frac{368}{315} e^{-4 \sigma ^2}\\*[1mm]
   +\frac{4}{5}\sigma^{-1}e^{-\sigma ^2} \Erf(\sigma )
   +\frac{32}{5} \sigma^{-1} e^{-2\sigma ^2} \Erf(\sigma )
   +\frac{4}{5 } \sigma ^{-1} e^{-3 \sigma ^2}
   \Erf(\sigma )\\*[1mm]
-\frac{32}{5} e^{-\sigma ^2}
   \Erf(\sigma )^2-\frac{32}{5} e^{-2 \sigma ^2}
   \Erf(\sigma )^2
    -\frac{16}{3}\sigma^{-1} e^{-\sigma ^2} \Erf(\sigma )^3,
\end{multline*}
which is translated by table \eqref{eq:I1}--\eqref{eq:I3}, without further ado, to
\begin{multline*}
V_4 = -\frac{152}{315}+\frac{68 \sqrt{2}}{105}-\frac{16
   \sqrt{3}}{35}-\frac{16 \log
   (2)}{5}+\frac{2 \log (3)}{5}\\*[0.375mm]+\frac{4}{5} \log
  (1+\sqrt{2})+\frac{32}{5} \log
  (1+\sqrt{3})-\frac{8 \pi }{15}-\frac{8\sqrt{2}}{5}  \arctan\frac{1}{2 \sqrt{2}}\\*[0.375mm]
   -\frac{16}{3}\frac{2}{\sqrt\pi}\int_0^\infty \sigma^{-1} e^{-\sigma ^2} \Erf(\sigma )^3\,d\sigma.
\end{multline*}
As argued at the end of the Appendix B, it is highly unlikely that the reminder integral has an elementary closed form expression. Also, its numerical value
\[
\frac{2}{\sqrt\pi}\int_0^\infty \sigma^{-1} e^{-\sigma ^2} \Erf(\sigma )^3\,d\sigma = 0.20145\,64675\,53825\,02025\cdots
\]
does not reveal any structure when put to the Inverse Symbolic Calculator\footnote{\url{http://wayback.cecm.sfu.ca/projects/ISC/ISCmain.html}}---whereas all the other integrals of terms in $\tilde g(\sigma)$ would have enjoyed a far better fate.
However, using the single integral representation in \eqref{eq:Vd}, a direct numerical integration gives easily
\[
V_4 = 1.48143\,26365\,21064\,74974\cdots,
\]
or being even more daring,
\[
V_{100} = 0.24625\,54841\,88745\,57533\cdots\,.
\]
\end{example}

\section{Applications: The Newton Potential and Force Field of a Cube}\label{sect:waldvogel}

As a quantity of interest in physics, the potential of a homogeneous rectangular cuboid $Q$
at a point $(y_1,y_2,y_3)$ in space---which we fix, to simplify notation, at the origin after performing an appropriate translation---is given by
\begin{equation}\label{eq:Vxyz}
V_0 := \iiint\limits_{Q}\frac{dx_1dx_2dx_3}{\sqrt{x_1^2+x_2^2+x_3^2}} =\frac{2}{\sqrt\pi}\int_0^\infty h_1(\sigma)h_2(\sigma)h_3(\sigma)\,d\sigma,
\end{equation}
where $h_j$ specifies the following function for index $j$, obtained from \eqref{eq:Erf}:
\[
h(\sigma) := \int_a^b e^{-\sigma^2x^2}dx = \left.\frac{\Erf(\sigma x)}{\sigma}\right|_{a}^b.
\]
The integrand in \eqref{eq:Vxyz} is thus (with bounds $a_j,b_j$ belonging to $\delta_j$)
\begin{equation}\label{eq:waldvogelintegrand}
h_1(\sigma)h_2(\sigma)h_3(\sigma) 
=\frac{\Erf(\sigma \delta_1) \Erf(\sigma \delta_2) \Erf(\sigma \delta_3)}{\sigma^3}\bigg|_{a_1}^{b_1} \bigg |_{a_2}^{b_2} \bigg |_{a_3}^{b_3};
\end{equation}
that is, an expression in terms of the form \eqref{eq:nosing} with $\kappa=1$, which is the only form listed in \eqref{eq:origerfterms} that is integrable at $\sigma=0$. Although not needed here to remove cancelling singularities, we still apply the renormalization algorithm of Section~\ref{sect:V} to get the standard integrals \eqref{eq:I1}--\eqref{eq:I3}: in this way, by applying rule \eqref{eq:rule4} first and rule \eqref{eq:rule3} next, we simply read off that\\*[-6mm]
\begin{multline*}
V_0 =\frac{2}{\sqrt{\pi}}\int_0^\infty\bigg(\sum_{\text{cyc:$\delta$}}
\Big( \delta_1\delta_2\,\sigma^{-1}e^{-\sigma^2(\delta_1^2+\delta_2^2)} \Erf(\sigma \delta_3)\\*[-4mm]
-\delta_1^3e^{-\sigma^2\delta_1^2}\Erf(\sigma\delta_2)\Erf(\sigma\delta_3)\Big)\bigg)\bigg|_{a_1}^{b_1} \bigg |_{a_2}^{b_2} \bigg |_{a_3}^{b_3}\,d\sigma.
\end{multline*}\\*[-3mm]
This is instantly translated by table \eqref{eq:I1}--\eqref{eq:I3} to Waldvogel's formula \eqref{eq:waldvogel},\footnote{
Hackbusch \cite[§4]{Hackbusch} grants the potential $V(y_1,y_2,y_3)$, equipped with additional monomial factors in the numerator of the integrand, just a short mention by saying that it “appears, e.g., in the collocation method and can be treated similarly”. He sketches the modifications of his method but does not offer any specific formulae. The Laplace transform technique was previously used by Ciftja \cite[Eq.~(32)]{Ciftja15} in 2015 for expressing the potential $V(y_1,y_2,y_3)$ in form of the single integral \eqref{eq:Vxyz} with an integrand
similar to \eqref{eq:waldvogelintegrand}. However, the subsequent explicit evaluation, ultimately yielding Waldvogel's for\-mula~\eqref{eq:waldvogel} in the form \cite[Eq. (14)]{Ciftja15}, extends over~4 additional pages in his paper.}
\[
V_0 =\Bigg(\sum_{\text{cyc:$\delta$}}
\bigg( \delta_1\delta_2 \arctanh\frac{\delta_3}{\rho_\delta} - \frac{\delta_1^2}{2}\arctan\frac{\delta_2\delta_3}{\delta_1\rho_\delta}\bigg)\Bigg)\Bigg|_{a_1}^{b_1} \Bigg |_{a_2}^{b_2} \Bigg |_{a_3}^{b_3}.
\]
Waldvogel had obtained this formula by evaluating, in a first step, the third component of the force field---initially given, as a derivative of the potential, by means of a double integral---in form of an expression with 24 terms. Only then, without any further explicit integration, by using Euler's theorem on homogeneous functions, he arrived at the potential $V_0$. 

It is entertaining to note that those 24 term expressions for the components of the force field of a homogeneous rectangular cuboid were obtained as early as 1822--30 by geodesist George Everest (of Mt Everest fame): he used them to estimate that the gravitational attraction of the Satpura Range, a tableland in Central India commencing as far as 30km off to the north, would have deflected a plumb line at Takal K'hera (Takarkheda) such that the zenith is thrown $5.1''$ southwards.\footnote{An estimate that Everest \cite[p.~104]{Everest} checked against another one, just $0.6''$ smaller, which he had obtained from matching a wide triangulation to a model of the equatorial bulge.} The force-field expressions, expanded in all their glory, occupy a whole page \cite[p.~97]{Everest} in his monumental “Account of the Measurement of an Arc of the Meridian between the Parallels of $18^\circ 3'$ and $24^\circ 7'$.”

Instead of reproducing Everest's force field of a homogeneous cuboid (which can, of course, easily be done using the method of this paper), we will discuss yet another problem taken from geophysics \cite{cubic}: namely, 
 the computation of the force field, once more evaluated at the origin, if the cuboid has a density variation in the third coordinate direction, modeled by a cubic polynomial. In analogy to \eqref{eq:FjVanilla} and \eqref{eq:Fj} we get, when looking at the  monomial $x_3^n$ of the polynomial density profile, a contribution to the force in the third coordinate direction in form of the quantity
\begin{equation}\label{eq:Fxyz}
H_n :=\iiint\limits_{Q}\frac{x_3^{n+1} \; dx_1dx_2dx_3}{(x_1^2+x_2^2+x_3^2)^{3/2}} =\frac{2}{\sqrt\pi}\int_0^\infty h_1(\sigma)h_2(\sigma)h_3^*(\sigma)\,d\sigma.
\end{equation}
Here $h_j$ is defined as above and, by reference to Lemma~\ref{lem:AB} and Section~\ref{sect:Fj}, taking the leading order $n=3$ of the cubic as an example,
\[
h_3^*(\sigma) := 2\sigma^2 \int_{a_3}^{b_3}x^4 e^{-\sigma^2x^2} \,dx =\left.\left(
\frac{3 \Erf(\sigma  x)}{2 \sigma ^3}
-\frac{3
   x e^{-\sigma ^2 x^2}}{2 \sigma
   ^2}
-x^3 e^{-\sigma ^2 x^2}
\right)\right|_{a_3}^{b_3}.
\]
In this way we obtain the integrand $h_1(\sigma)h_2(\sigma)h_3^*(\sigma)$ of \eqref{eq:Fxyz} in the form
\[
\frac{\Erf(\sigma \delta_1) \Erf(\sigma \delta_2) }{\sigma^2}\left(\frac{3 \Erf(\sigma  \delta_3)}{2 \sigma ^3}
-\frac{3
   \delta_3 e^{-\sigma ^2 \delta_3 ^2}}{2 \sigma
   ^2}
-\delta_3 ^3 e^{-\sigma ^2 \delta_3 ^2}\right)\Bigg|_{a_1}^{b_1} \Bigg |_{a_2}^{b_2} \Bigg |_{a_3}^{b_3},
\]
which is directly amenable to the renormalization algorithm of Section~\ref{sect:V}:
\begin{multline*}
H_3  = \frac{2}{\sqrt\pi}\int_0^\infty \frac{1}{4}\bigg(
\delta_1\delta_2\delta_3\rho_\delta^2 \; e^{- \sigma^2\rho_\delta} \\*[1mm]
- 2\delta_1\delta_2(\delta_1^2+\delta_2^2) \; \sigma^{-1}e^{-\sigma^2(\delta_1^2+\delta_2^2)} \Erf(\sigma\delta_3)\\*[0.5mm]
 +2\sum_{\text{cyc:$\delta$}}
\delta_1^5\; e^{-\sigma^2\delta_1^2} \Erf(\sigma\delta_2)\Erf(\sigma\delta_3)\bigg)\bigg|_{a_1}^{b_1} \bigg |_{a_2}^{b_2} \bigg |_{a_3}^{b_3}\;d\sigma.
\end{multline*}
This is instantly translated by table \eqref{eq:I1}--\eqref{eq:I3} to 
\begin{multline}\label{eq:garcia}
H_3 = \frac{1}{4}\bigg(\delta_1\delta_2\delta_3\rho_\delta
 - 2\delta_1\delta_2(\delta_1^2+\delta_2^2)\log(\delta_3 + \rho_\delta) \\*[1mm]
+ \sum_{\text{cyc:$\delta$}}\delta_1^4\arctan\frac{\delta_2\delta_3}{\delta_1\rho_\delta}\bigg)\bigg|_{a_1}^{b_1} \bigg |_{a_2}^{b_2} \bigg |_{a_3}^{b_3},
\end{multline}
where we have dropped the terms that are independent of $\delta_3$, namely
\[
2\delta_1\delta_2(\delta_1^2+\delta_2^2)\log\sqrt{\delta_1^2+\delta_2^2},
\]
since they would have cancelled when taking the difference  at the bounds of $\delta_3$. With similar easy we get the results for the other monomials of the cubic, or any other polynomial to begin with.

The closed form \eqref{eq:garcia} for $H_3$ agrees exactly with the expression \cite[Eq.~(20)]{cubic} which was obtained, using direct successive integration, by geophysicist Juan García-Abdeslem  in 2005 when studying a mass density varying with depth, modeled to follow a cubic profile---as given, e.g.,  by a fit to the density logging of 46 wells in  Green Canyon, located offshore Louisiana, in the Gulf of Mexico. As ever so often in this story, García-Abdeslem checked his closed form solution against a numerical method that he had published some 13 years earlier.


{\small \section*{\protect{\small Appendix}}
\section*{\protect{\small  A Two historic examples of multivariate integration by Laplace transform}}
We give two historic examples of the technique used in this paper, one in multivariate calculus attributed to Cauchy, and one from multivariate probability theory, due to Montroll in 1956. Impressive applications of the Laplace transform to a certain class of multiple integrals over polyhedra and ellipsoids can be found in \cite{MR1873902}.

\medskip

\begin{itemize}
\item In his course on differential and integral calculus, Fikhtengol'ts \cite[§650, Example 10]{Fichtenholz} discussed an example of  evaluating  a multivariate integral in terms of a single one by the Laplace transform technique, attributing the method to Cauchy (“\foreignlanguage{russian}{следуя Коши}”):
\begin{multline*}
\int_0^\infty\!\!\!\cdots\int_0^\infty \frac{x_1^{p_1-1}\cdots \;x_n^{p_n-1}e^{-(a_1x_1+\cdots+a_nx_n)}}{(b_0+b_1x_1+\cdots +b_nx_n)^q}\,dx_1\cdots \,dx_n \\*[5mm]= \frac{\GAMMA(p_1)\cdots \GAMMA(p_n)}{\GAMMA(q)}\int_0^\infty \frac{e^{-b_0 s}s^{q-1}}{(a_1+b_1 s)^{p_1}\cdots(a_b+b_n s)^{p_n}}\,ds \qquad (a_i,b_j,p_k,q>0).
\end{multline*}
This formula is obtained by inserting the Laplace transform (generalizing \eqref{eq:mean} and \eqref{eq:var})
\begin{equation}\label{eq:Gamma}
\frac{1}{t^q} = \frac{1}{\GAMMA(q)} \int_0^\infty s^{q-1}e^{-st}\,ds\qquad (t,q>0),
\end{equation}
changing the order of integration, and observing that \eqref{eq:Gamma} implies likewise
\[
\int_0^\infty x_j^{p_j-1} e^{-(a_j+b_js)x_j}\,dx_j = \frac{\GAMMA(p_j)}{(a_j+b_js)^{p_j}}\qquad (a_j,b_j,p_j,s>0).
\]
\item In 1956 Montroll \cite{MR88110} expressed the expected number $E$ of visits to the starting place of a symmetric random walk on the hypercubic lattice $\Z^n$, first, as the $n$-fold integral
\[
E= \frac{1}{\pi^n} \int_0^\pi\!\!\cdots\int_0^\pi \frac{d\phi_1\cdots \,d\phi_n}{1-(\cos\phi_1 + \ldots +\cos \phi_n)/n}.
\]
Next, inserting the $q=1$ case of the Laplace transform \eqref{eq:Gamma}, that is
\[
\frac{1}{t}  = \int_0^\infty e^{-st}\,ds\qquad (t>0),
\]
followed by changing the order of integration, and obeserving that
\[
I_0(\xi) = \frac{1}{\pi} \int_0^\pi e^{\xi \cos\phi}\,d\phi
\]
is the modified Bessel function of the first kind, Montroll obtained the single integral
\[
E= \int_0^\infty e^{-s} I_0(s/n)^n\,ds.
\] 
Using this, he proved that symmetric random walks are transient for $n\geq 3$, calculated the return probability for $n=3$ to 9 digits, and studied the limit of large dimensions. See \cite[§6.7]{MR2076374} for a generalization to biased random walks.
\end{itemize}}

{\small
\section*{\protect{\small B Proof of the two entries \eqref{eq:I2} and \eqref{eq:I3} of the integral table}}
For the sake of completeness, we prove the two integral formulae \eqref{eq:I2} and \eqref{eq:I3} that were found as entries in 
the table of integrals \cite{MR950173}. Let be $a,b,c\in\R$ with $a\neq 0$. 

\begin{itemize}
\item Using \eqref{eq:mean}, which is also the first entry \eqref{eq:I1} in the integral table of Section~\ref{sect:V}, we get by differentiation under the integral sign
\[
\frac{\partial }{\partial b} \frac{2}{\sqrt\pi}\int_0^\infty x^{-1} e^{-x^2 a^2}\Erf(b x)\,dx = \frac{2}{\sqrt\pi}\int_0^\infty e^{-x^2(a^2+b^2)}\,dx = \frac{1}{\sqrt{a^2+b^2}}.
\]
Since also
\[
\frac{\partial }{\partial b} \log(b +\sqrt{a^2+b^2}) = \frac{1}{\sqrt{a^2+b^2}},
\]
we obtain  \cite[Eq.~(2.8.5.8)]{MR950173} in form of the second log variant of \eqref{eq:I2}, that is
\begin{equation}\label{eq:I2new}
\frac{2}{\sqrt\pi}\int_0^\infty x^{-1} e^{-x^2 a^2}\Erf(b x)\,dx = \log(b +\sqrt{a^2+b^2}) - \log \sqrt{a^2}.
\end{equation}
\item Using \eqref{eq:var} we get by differentiation under the integral sign
\begin{multline*}
\frac{\partial^2 }{\partial b\,\partial c} \frac{2}{\sqrt\pi}\int_0^\infty  e^{-x^2 a^2}\Erf(b x)\Erf(c x)\,dx \\*[1mm]= \frac{2}{\sqrt\pi}\int_0^\infty x^2 e^{-x^2(a^2+b^2+c^2)}\,dx = \frac{1}{2(a^2+b^2+c^2)^{3/2}}.
\end{multline*}
Since also
\[
\frac{\partial^2 }{\partial b\,\partial c} \frac{1}{2a}\arctan\frac{b c}{a\sqrt{a^2+b^2+c^2}} = \frac{1}{2(a^2+b^2+c^2)^{3/2}},
\]
we obtain \cite[Eq.~(2.8.19.8)]{MR950173} in the form \eqref{eq:I3}, that is
\begin{equation}\label{eq:I3new}
\frac{2}{\sqrt\pi}\int_0^\infty  e^{-x^2 a^2}\Erf(b x)\Erf(c x)\,dx = \frac{1}{2a}\arctan\frac{b c}{a\sqrt{a^2+b^2+c^2}}.
\end{equation}
\item This technique reveals the limits of integration in finite terms of such expressions if there are three or more factors of $\Erf$ in the integrand. For instance, by \eqref{eq:I3new} 
we get
\begin{multline*}
\frac{\partial }{\partial a} \frac{2}{\sqrt\pi}\int_0^\infty  x^{-1} e^{-x^2}\Erf(x)^2\Erf(a x)\,dx  \\*[1mm]
= \frac{2}{\sqrt\pi}\int_0^\infty e^{-x^2(1+a^2)} \Erf(x)^2\,dx= \frac{\arctan\frac{1}{\sqrt{1+a^2}\sqrt{3+a^2}}}{2\sqrt{1+a^2}} 
\end{multline*}
and thus
\[
\frac{2}{\sqrt\pi}\int_0^\infty  x^{-1} e^{-x^2}\Erf(x)^2\Erf(a x)\,dx  = \int_0^a \frac{\arctan\frac{1}{\sqrt{1+ x ^2}\sqrt{3+ x ^2}}}{2\sqrt{1+x^2}} \,dx.
\]
Therefore, any closed form expression for the definite integral on the left, taken as a function of the para\-meter~$a$, would serve as a primitive of the 
integrand on the right. There is no apparent such primitive in  elementary finite terms, and as the sophisticated Risch--Davenport--Bronstein type algorithms for indefinite integration, at least as far as they are implemented in {\em Mathematica} and Maple, fail to provide one, such a primitive is highly unlikely to exist at all (since, if correctly implemented, a failure of those algorithms amounts as proof of the impossibility of indefinite integration in finite terms). In the specific case here, a rigorous proof could also be based on some differential Galois theory; it is left to the experts.
\end{itemize}}

\begin{acknowledgements} I would like to thank Nick Trefethen and Bengt Fornberg for their encouragement to write this paper---after they had seen a very short personal note about an easy way to obtain Fornberg's startling solution \eqref{eq:Fornberg} of Trefethen's two-cubes problem. Bengt brought Hackbusch's paper \cite{Hackbusch} to my attention, which was profoundly helpful in gaining a broader perspective on the subject; and Nick made valuable suggestions how to write a more engaging exposition (besides taking care of some of my English infelicities).

As a personal aside I would like to confess that this paper was stimulated by my “school-boy fascination” with special numerical expressions such as \eqref{eq:Fornberg}; thereby keeping Alf van der Poorten's remark from his foreword to Leonard Lewin's book \cite{MR618278} in mind: 
\begin{quote}
“To me it has occasionally seemed that mathematics threatens to be too serious a subject. Wonderful formulas are condemned as unimportant curiosities; or worse, as well known.
This book will assist in stemming any such trend.”
\end{quote}
I can only hope that this paper will assist in stemming any such trend, too.
\end{acknowledgements}

\bibliographystyle{spmpsci}		
\bibliography{paper.bib}   		

\end{document}